\documentclass{amsart}

\usepackage{latexsym,amsfonts}
\usepackage{amssymb}
\usepackage{amsmath}
\usepackage{amsthm}

\newtheorem{lemma}{Lemma}[section]
\newtheorem{theorem}[lemma]{Theorem}

\newtheorem{propo}[lemma]{Proposition}

\theoremstyle{remark}
\newtheorem{rem}[lemma]{Remark}

\def\Mat{\mathop{\rm Mat}\nolimits}          
\def\GL{\mathop{\rm GL}\nolimits}

\def\SO{\mathop{\rm SO}\nolimits}
\def\SL{\mathop{\rm SL}\nolimits}
\def\Sp{\mathop{\rm Sp}\nolimits}
\def\GL{\mathop{\rm GL}\nolimits}            
\def\PSL{\mathop{\rm PSL}\nolimits}          
\def\PSp{\mathop{\rm PSp}\nolimits}           
\def\PGL{\mathop{\rm PGL}\nolimits}           
\def\SU{\mathop{\rm SU}\nolimits}          
\def\PSU{\mathop{\rm PSU}\nolimits}          
\def\Sp{\mathop{\rm Sp}\nolimits}

\def\PGamL{\mathop{\rm P\Gamma L}\nolimits}

\newcommand{\tr}[1]{{\rm Tr}(#1)}
\def\char{\rm char\,}

\def\dim{\mathop{\rm dim}\nolimits }

\def\rad{\mathop{\rm rad}\nolimits }

\def\diag{\mathop{\rm diag}\nolimits}

\def\ovT{\overline{f}}

\def\eps{\varepsilon}

\def\om2{\omega^ {-\ell}}

\newcommand{\Z}{\mathbb{Z}}    
\newcommand{\F}{\mathbb{F}}    

\begin{document}

\title{Scott's formula and Hurwitz groups}

\author{M.A. Pellegrini}

\address{Dipartimento di Matematica e Fisica, Universit\`a Cattolica del Sacro Cuore, Via Musei 41,
I-25121 Brescia, Italy}
\email{marcoantonio.pellegrini@unicatt.it}

\author{M.C. Tamburini Bellani}
\address{Dipartimento di Matematica e Fisica, Universit\`a Cattolica del Sacro Cuore, Via Musei 41,
I-25121 Brescia, Italy}
\email{mariaclara.tamburini@unicatt.it}

 
\begin{abstract} 
This paper continues previous work, based on 
systematic use of a formula of L. Scott, to detect Hurwitz groups.
It closes the problem of determining the finite simple groups contained in
$\PGL_n(\F)$ for $n\le 7$ which are Hurwitz, where $\F$ is an algebraically closed field. 
For the groups $G_2(q)$, $q\ge 5$, and the Janko groups $J_1$
and $J_2$ it provides explicit $(2,3,7)$-generators. 
\end{abstract}

\maketitle

\section{Introduction}

Let $\F$ be  an algebraically closed field of any characteristic. Consider a finitely generated group $H$ and a 
representation $\Phi: H \rightarrow \GL(V)$, with representation space $V\cong \F^m$. For any subset $A$ of $H$ define: 
\begin{eqnarray*}
d_V^A & =& \dim\{ v\in V \mid \Phi(a)v=v, \textrm{ for all } a \in A\},\\ 
\hat d_V^A & =& \dim\{ v\in V \mid \Phi(a)^ T v=v, \textrm{ for all } a \in A\}.
\end{eqnarray*}
Moreover, for $A=\left\{a\right\}$, set $d_V^A=d_V^a$. 
By a special case of the well known formula of L. Scott \cite{S} (see also \cite{hb}),
if $H=\langle x,y\rangle$, then
\begin{equation}\label{Scott.gen}
d_V^ x+d_V^ y + d_V^ {xy}\leq m+ d_V^ H+\hat d_V^ H. 
\end{equation} 
This formula has a crucial impact on Hurwitz generation in low rank, as already suggested by Scott himself,
and made more evident by the work of A. Zalesski \cite{DTZ, VZ} and others after him. 

To illustrate this claim, it is convenient to introduce some terminology. A triple $(x,y,xy)$ in $\GL_n(\F)^3$  is called \emph{irreducible} whenever
$\left\langle x,y\right\rangle$ is an irreducible subgroup of $\GL_n(\F)$. It is called  
\emph{rigid} when it is irreducible  and equality holds in \eqref{Scott.gen}, with respect to  the conjugation action $\Phi$ 
of $H=\left\langle x,y\right\rangle$ on  $V=\Mat_n(\F)$. Finally $(x,y,xy)$  is said to be a $(2,3,7)$-triple 
whenever the projective images of $x,y,xy$ have respective orders $2,3,7$. Clearly all these properties are preserved by conjugation.

In \cite{TV1} were classified (up to multiplication by scalars)
the admissible similarity invariants, equivalently the admissible Jordan canonical forms (e.g., see \cite[3.10]{Jac}), of $x,y,xy$ in an irreducible  $(2,3,7)$-triple  $(x,y,xy)\in \GL_n(\F)^3$, $n\le 7$. It turns out that such triples  are all rigid if and only if $n\le 5$.
For $n=6,7$, the similarity invariants of the non-rigid ones (parametrized in \cite{V3} and \cite{TV2}) are  respectively: 
\begin{equation}\label{inv}
t^ 2+1, \; t^ 2+1,\; t^ 2+1; \qquad t^3-1,\; t^3-1; \qquad t^ 6+t^ 5+t^ 4+t^ 3+t^2+t+1,
\end{equation}
for $n=6$ (see  \cite[1.1.1 page 349 and 2.1.1 page 350]{TV1}), and  
\begin{equation}\label{inv7}
t+1, \;t^ 2-1, \; t^ 2-1,\; t^ 2-1; \qquad t-1,\;t^3-1,\; t^3-1; \qquad t^ 7-1,
\end{equation}
for $n=7$ (see \cite[1.1 page 351]{TV1}).

An irreducible triple of type \eqref{inv} generates a symplectic group. The irreducibility of a triple of type 
\eqref{inv7} forces $\char \F\neq 2$ and the corresponding group is orthogonal. The proof of these facts uses techniques based
on Scott's formula (see Lemma \ref{lem:symp}).

In this paper we continue the analysis started in \cite{TV1} and \cite{TV2} with respect to the  non-rigid triples. In particular we prove the  following two results:

\begin{theorem}\label{p7}
Suppose that $\F$ is an algebraically closed field. Let $H_7$ 
be an irreducible subgroup of $\SL_7(\F)$ generated by a non-rigid 
$(2,3,7)$-triple, i.e. a triple with similarity invariants \eqref{inv7}. Then $\char \F\neq 2$
and  $H_7$ is a subgroup of  $G_2(\F)$.
\end{theorem}

\begin{theorem}\label{main}
Suppose that $\F$ is an algebraically closed field of characteristic $2$. Let $H_6$ 
be an irreducible subgroup of $\SL_6(\F)$ generated by a non-rigid 
$(2,3,7)$-triple, i.e., a triple with similarity invariants \eqref{inv}. Then $H_6$ is a subgroup of  $G_2(\F)$.
\end{theorem}

These two theorems allow to complete the classification of the finite simple groups, admitting 
absolutely irreducible projective representations of degree $n\le 7$, which are Hurwitz. 
Among such groups, those which are generated by rigid triples, in particular those which admit irreducible projective representations of degree $n\leq 5$, have already been classified (see \cite{TV1}).
Thus we are left with those  generated by irreducible non-rigid triples. 
For $n=6$ and $\char \F\neq 2$, by Lemma \ref{lem:symp}(i) they are contained in $\Sp_6(q)$, $q$ odd,
where $q$ stands for the order of a finite field.
Otherwise,  they are contained in $G_2(q)$ by Theorems \ref{p7}  and \ref{main}.
Thus the classification can be completed analyzing  $\Sp_6(q)$,
$G_2(q)$ and their subgroups. 

The group  $\PSp_6(q)$ is not Hurwitz for $q=3$ and for $q$ even (see \cite{VZ}; for $q$ even this fact also follows from Theorem \ref{main}).
Hurwitz generators of $\PSp_6(q)$, $q\ge 5$ odd, are given in \cite{TV}.
On the other hand, by an old, non-constructive result  of G. Malle, $G_2(q)$ itself is a Hurwitz group if, and 
only if, $q\geq 5$. 
Explicit $(2,3,7)$-generators for the  Ree groups $^2G_2\left(3^{2a+1}\right)$ can be found in \cite{T}.

For the reader's convenience
we state  the full classification.
Given an integer $u$ and a prime $p$, coprime to $u$, we denote by $\mathrm{o}_u(p)$ the order of 
$p\pmod u$, i.e. the order of $p+u\Z$ seen as an element of the group $\left(\Z / u\Z\right)^\ast$.

\begin{theorem}\label{Hurwitzlist}
Let $G$ be a finite simple group admitting an absolutely irreducible projective representation of degree $n\le 
7$. Then $G$ is Hurwitz if, and only if, it is isomorphic to one of the following:
\begin{itemize}
\item[{\rm (i)}] $\PSL_2(p)$ when $p\equiv 0,\pm 1\pmod 7$, $\PSL_2(p^{3})$ when $p\equiv \pm 2, \pm 3\pmod 7$
\cite{Mac}; 
\item[{\rm (ii)}] $\PSL_5(p^{n_5})$ when $p\ne 5,7$  and $n_5$ is odd, $\PSU_5(7^4)$ and $\PSU_5(p^{n_5})$ when $p\ne 5,7$ 
and $n_5$ is even, where $n_5$ is $\mathrm{o}_5(p)\cdot \mathrm{o}_7(p^2)$  \cite{TZ};
\item[{\rm (iii)}] $\PSL_6(p^{n_6})$ when $p\ne 3$ and $n_6$ is odd, $\PSU_6(p^{n_6})$ when $p\ne 3$ and
$n_6$ is even, where $n_6=\mathrm{o}_9(p)$ \cite{TV1};
\item[{\rm (iv)}] $\PSL_7(p^{n_7})$ when $p\ne 7$ and $n_7$ is odd, $\PSU_7(p^{n_7})$ when $p\ne 7$ and
$n_7$ is even, where $n_7=\mathrm{o}_{49}(p)$ \cite{TV1};
\item[{\rm (v)}] $\PSp_6(q)$ when $q\geq 5$ is odd \cite{TV}; 
\item[{\rm (vi)}] $G_2(q)$ when $q\geq 5$ and $^2G_2\left(3^{2a+1}\right)$ for all $a\geq 1$ \cite{Ma};
\item[{\rm (vii)}]  $J_1$ \cite{Sah}, $J_2$ \cite{FRu}.
\end{itemize}
\end{theorem}
The groups in items {\rm (i)}-{\rm (iv)} of previous theorem are all generated by rigid triples.
Clearly rigidity restricts the isomorphism types: at most one for fixed $p$ and $n$.
For an interesting discussion of this aspect we refer to \cite{M}.

In the last Section  we provide explicit Hurwitz generators of $G_2(q)$, $q\geq 5$, 
and of the Janko groups $J_1$ and  $J_2$, respectively over $\F_{11}$ and $\F_4$.

\subsection*{Acknowledgments}

We are indebted to Maxim Vsemirnov for useful suggestions and discussions.

\section{Preliminary results}\label{prelim}

Let  $(x_n,y_n,x_n y_n)$ be an irreducible triple in $\GL_n(\F)^3$, where $\F$ is an algebraically closed field,
and set $H_n=\left\langle x_n,y_n\right\rangle$. We consider the diagonal action of $\GL_n(\F)$ on the space $\F^n\otimes \F^n$,
identified with $\Mat_n(\F)$, namely the action $B\mapsto gBg^T$ for all $B\in \Mat_n(\F)$, $g\in\GL_n(\F)$.  
The symmetric square $S=S(\F^n)$ and the exterior 
square power $\Lambda^ 2=\Lambda^2(\F^n)$ can be identified, respectively, with the spaces
of symmetric matrices
and antisymmetric matrices with zero-diagonal. The spaces $S$ and  $\Lambda^ 2$, being $\GL_n(\F)$-invariant,
give rise to representations of $H_n$ on $\F$ of respective degrees $\frac{n(n+1)}{2}$ and
$\frac{n(n-1)}{2}$.
By \cite[Lemma 1(i)]{TV1},  in the first representation one has
\begin{equation}\label{dhat}
\hat d_{S}^ {H_n}\le d_{S}^ {H_n} \le 1.
\end{equation}
Here the  notation is the same defined at the beginning of  the Introduction.

\begin{lemma}\label{lem:symp}
In the above notation, the following facts hold:
\begin{itemize}
\item[\rm{(i)}] If $n=6$ and  $(x_6,y_6,x_6y_6)$ is an irreducible triple in $\GL_6(\F)^3$ with similarity invariants \eqref{inv},
then $H_6=\langle x_6,y_6\rangle$ is contained in $\Sp_6(\F)$.
\item[\rm{(ii)}] If $n=7$ and  $(x_7,y_7,x_7y_7)$ is an irreducible triple in $\GL_7(\F)^3$ 
with similarity invariants \eqref{inv7}, then $\char\F \neq 2$ and $H_7=\langle x_7,y_7\rangle$ is contained in $\Omega_7(\F)$.
\end{itemize}
\end{lemma}

\begin{proof}
(i) If $\char \F\neq 2$ we consider the action 
of $H_6$ on $\Lambda^ 2=\Lambda^2(\F^6)$, as in \cite{V3}.
Then $\F^6\otimes \F^6= S\oplus \Lambda^2$, whence
$\hat d_{\Lambda ^2}^{H_6}=d_{\Lambda ^2}^ {H_6^T}$. From $d_{\Lambda ^2}^ {x_6}= 9$, $d_{\Lambda ^2}^{y_6}\ge 5$, $d_{\Lambda ^2}^ {x_6y_6}\ge 3$
and  $\dim \Lambda^2 =15$,  we get $d_{\Lambda ^2}^ {H_6}+ \hat d_{\Lambda ^2}^{H_6}\ge 2$.
Since  $d_{\Lambda ^2}^ {H_6}\le  1$ and $d_{\Lambda ^2}^ {H_6^T}\le 1$, by the irreducibility of  $H_6$,
it follows $d_{\Lambda ^2}^ {H_6}= d_{\Lambda ^2}^ {H_6^T}= 1$. Hence $H_6$ fixes a non-zero alternating 
antisymmetric form $J$, which is non-degenerate by the irreducibility of $H_6$. 
We conclude $H_6\le \Sp_6(\F)$.

Now assume  $\char \F=2$. Then $d_{S}^ {x_6}= 12$, $d_{S}^{y_6}= 7$ and $d_{S}^ {x_6y_6}= 3$.
Since $\dim S=21$, from  \eqref{dhat} and Scott's formula \eqref{Scott.gen}  it follows  either 
$d_S^{H_6}= \hat d_S^ {H_6}= 1$ or $d_S^ {H_6}=1$ and $\hat d_S^{H_6}=0$.
In the first case $H_6\le \Omega_6(\F)\le \Sp_6(\F)$, in the second $H_6\le \Sp_6(\F)$  by \cite[Lemma 1(ii)-(iii)]{TV1}.

(ii) We have  $d_{S}^{x_7}= 16$, $d_{S}^{y_7}= 10$, $d_{S}^ 
{x_7y_7}= 4$.  Since $\dim S=28$, 
from  \eqref{Scott.gen} we get $d_S^{H_7}+\hat d_S^ {H_7}\ge 2$. Again \cite[Lemma 1]{TV1} gives 
$\char \F\neq 2$,
$d_S^{H_7}= \hat d_S^ {H_7}= 1$, and $H_7$ orthogonal. Since $H_7$ is $(2,3,7)$-generated, it coincides with its derived subgroup and so it is contained in $\Omega_7(\F)$.\hfill $\square$
\end{proof}

When $\char \F=2$, $\Sp_{6}(\F)$ is isomorphic to the orthogonal group $\Omega_7(\F)$, by a classical result of Dieudonn\'e \cite{Dieu}. Indeed, call  $Q$ the quadratic form  which defines $\Omega_{7}(\F)$,  call
$g$ the associated  symmetric bilinear form  of rank $6$ and 
$\left\langle v_0\right\rangle$ the radical of $g$. Then $\Omega_{7}(\F)$ fixes $\left\langle v_0\right\rangle$
and induces a symplectic group on $\frac{\F^{7}}{\left\langle v_0\right\rangle}$.
Since every symplectic transformation has determinant $1$,  $\Omega_{7}(\F)$ fixes $v_0$.
Hence, with respect to any basis of $\F^{7}$ having $v_0$ as first vector, $\Omega_{7}(\F)$ consists of
matrices 
\begin{equation}\label{Omega7}
h=\begin{pmatrix} 1 & a_{h}^T \\ 0 &  h_{6} \end{pmatrix},\ 
a_h\in \F^{6}, \ h_{6}\in \Sp_{6}(\F).
\end{equation}
The map $h\mapsto h_{6}$ is an isomorphism. We use its inverse $h_6\mapsto h$. The image $(\tilde x_7, \tilde y_7, \tilde x_7\tilde y_7)$ of $(x_6,y_6, x_6y_6)$ in $\Omega_{7}(\F)$ has
similarity invariants  \eqref{inv7}. Since the eigenspace of $\tilde x_7\tilde y_7$ relative to $1$ has dimension 
$1$, it must coincide with $\left\langle v_0\right\rangle$. Let 
$\mathcal{B}=\{v_0, v_1, v_2,v_{-3}, v_{-1},$ $v_{-2}, v_{3}\}$ be a basis  of 
eigenvectors of $\tilde x_7\tilde y_7$, with $(\tilde x_7\tilde y_7)v_i=\eps^{-i}v_i$ for $i=0,\pm 1, \pm 2,$ $\pm 3$,
where $\eps\in \F$ is a primitive $7$-th root of $1$.
With respect to $\mathcal{B}$ the Gram matrix $J$
of the bilinear form $g$ fixed by $\left\langle \tilde x_7,\tilde y_7\right\rangle$ is
$J=\begin{pmatrix} 0& 0\\0&J_6\end{pmatrix}$ with $J_6$ fixed
by $\left\langle x_6,y_6\right\rangle$. 
Multiplying $v_1,v_2, v_{-3}$ by scalar multiples, if necessary, 
we have $J_6=\begin{pmatrix} 0& I_3\\I_3&0\end{pmatrix}$.
Hence we may suppose:
\begin{equation}\label{shape7}
\tilde x_7=\begin{pmatrix} 1 & a_{x}^T \\ 0 &  x_6  \end{pmatrix}, \quad \tilde y_7 =
\begin{pmatrix} 1 &a_{x}^T x_6y_6\\
 0 &  y_6  \end{pmatrix},\quad 
\tilde x_7\tilde y_7 =\begin{pmatrix} 1 &
0\\ 0 &  x_6y_6  \end{pmatrix} .
\end{equation}
The conditions $\tilde x_7^2=I$, $J_6x_6=x_6^TJ_6$ and the choice of $\mathcal{B}$ give 
$a_{x}^Tx_6=a_{x}^T$ and
\begin{equation}\label{shapex6}
\left.\begin{array}{c}
x_6=
\begin{pmatrix} A &B\\
C&A^T \end{pmatrix},\ B=B^T, C=C^T, \\
x_6y_6= \diag\left(\eps^{-1},\eps^{-2},\eps^{3},\eps^{1},\eps^{2},\eps^{-3} \right).
\end{array}\right.
\end{equation}

\section{Proofs of Theorems \ref{p7} and \ref{main}}

The  group $G_2(\F)$ can be defined as the subgroup of $\GL_7(\F)$ fixing a particular alternating trilinear form, called the Dickson form (see \cite{A}). For the reader's convenience, we recall  some basic facts and definitions concerning alternating $m$-linear forms on $V=\F^n$  ($n\geq m$). 

An $m$-linear map $f: V^m\rightarrow \F$ is said to be alternating if
$f(v_1,\ldots,v_m)=0$ whenever $v_i=v_j$ for some $i\neq j$.
We define 
$$\wedge^m(V)=\frac{\otimes_m V}{\langle v_1\otimes\ldots\otimes v_m: v_i=v_j  \textrm{ for some } i\neq j\rangle}.$$
By $v_1\wedge \ldots \wedge v_m$ we denote the image of  $v_1\otimes \ldots \otimes v_m$ under the canonical projection of $\otimes_m V$ onto $\wedge^m(V)$.  
The spaces $\wedge^m(V)$ and $\wedge^m(V^\ast)$, where $V^\ast$ denotes the dual of $V$, are dual to each other,
under the bilinear map
$$(v_1^\ast\wedge \ldots\wedge v_m^\ast, u_1\wedge \ldots \wedge u_m)\mapsto \det \begin{pmatrix}
v_k^T u_j
\end{pmatrix}$$
(e.g., see \cite[Chapter XIX]{Lang}). Hence, the space of alternating $m$-linear forms on $V$  can be
identified with the skew-symmetric tensor product $\wedge^m(V^\ast)$ acting on $V^m$  via:
\begin{equation}
(v_1^\ast \wedge\ldots\wedge
v_m^\ast)(u_1,\ldots,u_m)= \det \begin{pmatrix}
v_k^T u_j
\end{pmatrix}.
\end{equation}
We may define a representation $\Phi: \GL(V)\rightarrow \GL\left(\wedge^m(
V^\ast)\right)$ by setting:
$$[\Phi(g)(v_1^\ast \wedge \ldots\wedge
v_{m}^\ast)](u_1,\ldots,u_m)=(v_1^\ast \wedge\ldots\wedge
v_{m}^\ast)(g^{-1} u_1,\ldots,g^{-1}u_m),$$
namely
$$
\Phi(g)(v_1^\ast \wedge\ldots\wedge
v_{m}^\ast)=\left(g^{-T}v_1\right)^\ast \wedge \ldots\wedge
\left(g^{-T} v_m\right)^\ast.
$$ 

We claim that $\Phi(g^T)=\Phi(g)^T$ for all $g$. In particular, the representation $g\mapsto \Phi\left(g^{-T}\right)$
is the dual of $\Phi$.
Indeed, let $\{e_1,\ldots,e_n\}$ and $\{e_1^\ast,\ldots,e_n^\ast\}$
denote respective bases for $V$ and for the dual $V^\ast$ of $V$. 
Setting
$$\mathcal{B}=\left\{e_{i_1}^\ast \wedge\ldots\wedge e_{i_m}^\ast \mid\;  1\leq
i_1< i_2<
\ldots < i_m\leq n\right\}$$
we obtain a basis of $\wedge^m(V^\ast)$.
By the above
$$
\Phi(g)\left(e_{i_1}^\ast \wedge\ldots\wedge
e_{i_m}^\ast\right)=
\left( g^{-T} e_{i_1}\right)^\ast \wedge \ldots\wedge
\left( g^{-T}e_{i_m}\right)^\ast$$ 
$$=\sum_{1\le k_1\le\dots \le k_m\le n} \det \left(\gamma_{i_1,\dots ,i_m}^{k_1,\dots
,k_m}\right)\left(e_{k_1}^\ast \wedge\ldots\wedge e_{k_m}^\ast\right),
$$
where  $\gamma_{i_1,\dots ,i_m}^{k_1,\dots ,k_m}$ is the submatrix
obtained from $g^{-T}$ considering  rows $i_1,\dots ,i_m$ and columns $k_1,\dots
,k_m$.
The same calculation leads to
$$
\Phi(g^T)\left(e_{i_1}^\ast \wedge\ldots\wedge e_{i_m}^\ast\right) =
\sum_{1\le k_1\le\dots \le k_m\le n} \det \left(\gamma_{k_1,\dots ,k_m}^{i_1,\dots
,i_m}\right)\left(e_{k_1}^\ast \wedge\ldots\wedge e_{k_m}^\ast\right).$$
Identifying $\Phi(g)$ and $\Phi\left(g^T\right)$ with their matrices with respect to
$\mathcal{B}$
we have $\Phi\left(g^T\right)=\Phi(g)^T$, whence our claim.

When $m=3$ and $n=7$ we may take a basis $\{v_0,v_1,v_2,v_3,v_{-1},v_{-2},v_{-3}\}$ and consider the following alternating trilinear form, called a Dickson form:
$$v_0^\ast\wedge v_1^\ast\wedge v_{-1}^\ast+v_0^\ast\wedge v_2^\ast\wedge v_{-2}^\ast
+v_0^\ast\wedge v_3^\ast\wedge v_{-3}^\ast +v_1^\ast\wedge v_2^\ast\wedge v_{-3}^\ast+
v_{-1}^\ast\wedge v_{-2}^\ast\wedge v_{3}^\ast.$$
\smallskip

Our proofs of Theorems \ref{p7} and \ref{main} are based on key results of Aschbacher \cite[Theorem 5]{A}. 
The case $n=7$ turns out to be much easier to handle than the case $n=6$. For this reason we consider it firstly.

\begin{proof}[Proof of Theorem \ref{p7}]
Let $H_7$ be an irreducible subgroup of $\SL_7(\F)$, generated by a 
non rigid $(2,3,7)$-triple $(x_7,y_7,x_7y_7)$. Its similarity invariants are
those in \eqref{inv7} by \cite{TV2}.  
Moreover char $\F\ne 2$ and
$H_7\le \Omega_7(\F)$ (see Lemma \ref{lem:symp}). 
Consider the space  $V=\F^7$ with basis $\{e_1,\ldots,e_7\}$. 
As done before, we may identify the 
space of the alternating trilinear forms defined on $V$ with the space $\wedge^3
(V^\ast)$ and consider the representation $\Phi: \GL(V)\rightarrow \GL\left(\wedge^3
(V^\ast)\right)$ previously introduced. 
We apply formula \eqref{Scott.gen} to this action:
$$d_{\Lambda^ 3(V^\ast)}^{x_7}+d_{\Lambda^ 3(V^\ast)}^{y_7}+d_{\Lambda^ 3(V^\ast)}^{x_7y_7} \leq 35
+ d_{\Lambda^ 3(V^\ast)}^{H_7}+\hat d_{\Lambda^ 3(V^\ast)}^{H_7}.$$

Considering the  canonical forms of $x_7$, $y_7$ and $x_7y_7$ we obtain $d_{\Lambda^ 3(V^\ast)}^{x_7}=
19$, $d_{\Lambda^ 3(V^\ast)}^{y_7}= 13$ and $d_{\Lambda^
3(V^\ast)}^{x_7y_7}=5$, whence 
$$d_{\Lambda^ 3(V^\ast)}^{H_7}+\hat d_{\Lambda^ 3(V^\ast)}^{H_7} \geq 2.$$
As $\Phi(h)^T=\Phi(h^T)$ for all $h\in H_7$ we obtain that $ \hat d_{\Lambda^
3(V^\ast)}^{H_7}= d_{\Lambda^ 3(V^\ast)}^{H_7^T}$.
Further, since $H_7$ and $H_7^T$ are absolutely irreducible, by \cite[Theorem 5(2)]{A}, 
$$0\leq d_{\Lambda^ 3(V^\ast)}^ {H_7}\le 1, \quad  0\le d_{\Lambda^ 3(V^\ast)}^{H_7^T}\leq 1.$$
It follows that  $d_{\Lambda^ 3(V^\ast)}^ {H_7}= \hat d_{\Lambda^ 3(V^\ast)}^{H_7}=1$.
Thus $H_7$ fixes a non-zero alternating trilinear form which, by \cite[Theorem 5(5)]{A}, must be a Dickson form.
We conclude that $H_7\leq G_2(\F)$.\hfill $\square$
\end{proof}

We now turn to the proof of Theorem \ref{main} and so, from now on, we assume $\char \F=2$ and keep notation  \eqref{shape7}. 
A fundamental step of the proof is that, whenever $H_6=\left\langle x_6,y_6\right\rangle$ is absolutely
irreducible,
then $H_7=\left\langle\tilde x_7,\tilde  y_7\right\rangle$
fixes a non-zero alternating trilinear form.
This will be done in Proposition \ref{Scott}, that requires some preliminary facts established in Proposition \ref{notril6} and Lemmas \ref{orto} and \ref{invSub}.

\begin{propo}\label{notril6}
Let $K$ be an irreducible subgroup  of $\SL_6(\F)$ generated by 
a $(2,3,7)$-triple whose similarity
invariants are \eqref{inv}. 
Then $K$ does not fix any non-zero alternating trilinear form.
\end{propo}

\begin{proof}
By the above discussion we may suppose that $K$ is generated by $x=x_6$, $y=y_6$ and $z=xy$ as in \eqref{shapex6}.
Proceeding by way of contradiction, let $f$ be a non-zero alternating trilinear form fixed  by $K$.
We have  $f(z^ {-1} v_i, z^{-1} v_j, z^ {-1} v_k)=
\eps^ {i+j+k}f(v_i,v_j,v_k)$. Since $f$ is fixed by $z$  we get $f(v_i,v_j,v_k)=0$ whenever 
$i+j+k\not\equiv 0\pmod 7$. Thus:
$$f=\lambda (v_1^\ast \wedge v_2^\ast \wedge v_{-3}^ \ast)+\mu (v_{-1}^\ast \wedge v_{-2}^\ast \wedge v_{3}^ \ast),$$
for some $\lambda,\mu\in \F$. Assume $\mu =0$. In this case $\langle v_{-1},v_{-2},v_{3}\rangle\leq \rad(f)$ and so
$\rad(f)$ is a non-trivial subspace of $\F^ 6$ fixed by $K$.
Since $K$ is irreducible, we obtain $\rad(f) =\F^6$ and so $f=0$, a contradiction. The same holds if $\lambda=0$. 
It follows that both $\lambda$ and $\mu$ are non-zero and, multiplying by $\lambda^{-1}$, we may assume:
\begin{equation}\label{trilK}
f=v_1^\ast \wedge v_2^\ast \wedge v_{-3}^ \ast+\rho\left( v_{-1}^\ast \wedge v_{-2}^\ast \wedge v_{3}^ \ast\right),\ \rho\ne 0
\end{equation}
i.e., $f\left(v_{1},v_{2},v_{-3}\right)=1$, $f\left(v_{-1},v_{-2},v_{3}\right)=\rho$, 
$f\left(v_{i},v_{j},v_{k}\right)=0$ otherwise.
Write 
$$x=(a_{i,j}),\quad i,j\in\{1,2,-3,-1,-2,3\}.$$

\noindent \textbf{Case 1.} Suppose that $A$ in \eqref{shapex6} is diagonal.
From $f(x v_1, v_2,v_{-3} )=f(v_1, x v_2,$ $x v_{-3})$ we get
$a_{1,1}=a_{2,2}\cdot a_{-3,-3}$. Similarly, from $f(v_1,  x v_2,v_{-3})=
f(x v_1, v_2, $ $x v_{-3})$ and 
$f(v_1,v_2, x v_{-3})=f( x v_1,  x v_{2}, v_{-3})$ we get, respectively, 
$a_{2,2}=a_{1,1}\cdot a_{-3,-3}$ and $a_{-3,-3}=a_{1,1}\cdot a_{2,2}$. It follows that either $a_{1,1}=a_{2,2}=a_{-3,-3}=0$ or $a_{1,1}=a_{2,2}=a_{-3,-3}=1$. 
In the first case  $A=0$ and it is clear from \eqref{shapex6} that $y=x(xy)$ cannot have order $3$. 
In the second case, $A=I$, whence the contradiction
$1=\tr{x(xy)}=\tr{y}=0$.
\medskip

\noindent\textbf{Case 2.} Some non-diagonal entry of $A$ is not zero. We may suppose that it is 
in the first column, up to the permutation $\left(v_{1}v_{i}\right)\left(v_{-1}v_{-i}\right)$ for some $i=2,-3$.
We claim that $a_{2,1} \ne 0$ if, and only if, $a_{-3,1}\ne 0$.
Indeed, suppose $a_{2,1} \ne 0$. The subspace $S_1=\{s\in \F^ 6\mid f(v_1, x v_1, s)=0\}$ is 
$x$-invariant and contains $\langle v_{-1},v_{-2},v_3,v_1\rangle$. 
If $a_{-3,1}=0$, then $v_{-3}\not\in S_1$. As $v_2 \in S_1$ the space
$S_1$ has dimension $5$ and is fixed by $K$, a contradiction. Thus $a_{-3,1}\ne 0$ and the same 
argument shows the opposite implication.  Conjugating $x$ and $z$ by 
$\diag\left(a_{2,1}^{-1}\cdot a_{-3,1}^{-1}, a_{2,1}, a_{-3,1}, \ a_{2,1}\cdot a_{-3,1}, a_{2,1}^{-1}, a_{-3,1}^{-1}\right)$
their shapes are preserved as well as the condition $f\left(v_1, v_2,v_{-3}\right)=1$.
Hence we may suppose
$a_{2,1}=a_{-3,1}=1$ (and $a_{-1,-2}=a_{-1,3}=1$ by \eqref{shapex6}). 

For $j\in \left\{-1, -2,3\right\}$ we obtain:
$$0 = f(x v_1,  v_1, v_j)=f(v_1, x v_1, x v_j)=a_{-3,j}+ a_{2,j},$$
whence $a_{-3,-1}=a_{2,-1}$,  $a_{-3,-2}=a_{2,-2}$ and  $a_{-3,3}=a_{2,3}$. 
After these substitutions, for $j\in \left\{1, 2,-3\right\}$ we get
$$
\rho\thinspace a_{-1,j}=f( x v_j, v_{-2}, v_3) = f( v_j, x v_{-2}, x v_3) =0,$$
whence $a_{-1,1}=a_{-1,2}=a_{-1,-3}=0$ (and $a_{-2,1}=a_{3,1}=0$).
Now $f(x v_3, v_{-2}, v_j)$ $=f(v_3,  xv_{-2},  x v_j) $ for $j=2,-3$
gives 
$\rho\thinspace a_{-2,j}=0$, i.e., $a_{-2,2}=a_{-2,-3}=a_{3,2}=0$. Finally  
$f( x v_{-2}, v_{3}, v_{-3})=f(v_{-2},  xv_{3},  x v_{-3})$ gives $a_{3,-3}=0$.

We conclude that $\langle v_1,v_2,v_{-3} \rangle$  is an $x$-invariant subspace,
hence a $K$-invariant subspace, a contradiction.
\hfill $\square$ \end{proof}

\begin{rem}\label{quasidet} 
In characteristic $2$, a perfect irreducible subgroup of $\SL_6(\F)$ having 
an involution $x$ with similarity invariants 
$t^2+1,\ t^ 2+1,\ t^ 2+1$, cannot be contained in $\SO_6(\F)^\prime= \Omega_6(\F)$. Indeed the dimension of the fixed points space 
of $x$ is $3$, whence the quasideterminant of $x$ is, by definition, $(-1)^3=-1$. On the contrary, the quasideterminant of an involution in $\Omega_6(\F)$ is $1$ (see \cite[p. xii]{Atlas}).
\end{rem}

\begin{lemma}\label{orto} 
If $H_6$ is  irreducible, $\langle v_0\rangle$
has no $H_7$-invariant complement.
\end{lemma}

\begin{proof}
Any complement should coincide with the space $W=\left\langle v_1, \dots , v_{3}\right\rangle$
generated by the eigenvectors of $\tilde x_7\tilde y_7$ relative to the eigenvalues $\ne 1$. 
This happens only if $a_{x}^T =0$ in \eqref{shape7}. In this case $H_6$ preserves 
the restriction $\bar Q=Q_{|W}$, where $Q$ is the quadratic form fixed by $H_7$.
The symmetric form associated to $\bar Q$ has Gram matrix $J_6$, as above,
which is non-degenerate.
Hence $\bar Q$ is a quadratic form.
This implies that $H_6$ is contained in $\Omega_6(\F)$, in contrast with Remark \ref{quasidet}.

\hfill $\square$
\end{proof}

\begin{lemma}\label{invSub}
If $H_6$ is irreducible, then $\langle v_0\rangle$ is the only proper $H_7$-invariant subspace
of $\F^7$. Similarly $W=\langle v_{1}, v_{2}, v_{-3}, v_{-1}, v_{-2}, v_3\rangle$ is the
only proper
$H_7^T$-invariant subspace
of $\F^7$.
\end{lemma}

\begin{proof}
Let $U$ be a proper $H_7$-invariant subspace.
By the irreducibility of $H_6$ we have either $U\cap W=0$
or  $U\cap W=U$. In the first case $U$ has
dimension $1$.
Hence the perfect group $H_7$ induces the identity on $U$. This gives  $U=\langle
v_0\rangle$,
the eigenspace of $\tilde x_7\tilde y_7$ relative to $1$. The second case cannot arise, since $U$ would be an
$H_7$-invariant complement of $\langle v_0\rangle$, in contrast with the previous Lemma.
The second part of the statement follows noting that $\langle v_0^T\rangle$ is the
eigenspace 
of $(\tilde x_7\tilde y_7)^T$ relative to $1$ and using the fact that $H_6^T$ is not orthogonal by Remark 
\ref{quasidet}.
\hfill $\square$ \end{proof}

We are now ready to prove the following result on which Theorem \ref{main} is based.

\begin{propo}\label{Scott}
If $H_6$ is  irreducible, then
the group $H_7$  fixes a non-zero alternating trilinear form on $V$. 
\end{propo}

\begin{proof}
Proceeding as in the proof of  Theorem \ref{p7}, we take the action of $H_7$ on the space
of trilinear forms, that we identify with $\Lambda^3(V^*)$, obtaining  
$$d_{\Lambda^ 3(V^\ast)}^ {H_7}+\hat d_{\Lambda^ 3(V^\ast)}^{H_7}=d_{\Lambda^ 3(V^\ast)}^
{H_7}+d_{\Lambda^ 3(V^\ast)}^{H_7^T} \geq 2.$$
It suffices to show that $d_{\Lambda^ 3(V^\ast)}^{H_7^T}\le 1$.
So let $\tilde f$ be a non-zero alternating trilinear form on $V$
fixed by $H_7^ T$. 
Denote by $W$ the $6$-dimensional subspace of $V$ fixed by $H_7^ T$.
The restriction $\tilde f|_W$ of $\tilde f$ to $W$ is $ H_6^ T$-invariant and so it must
be the zero form, by Proposition \ref{notril6}. 
Furthermore, the kernel of the function $\Psi: V \rightarrow \Lambda^ 2 (V^ \ast)$,
defined
as $\Psi(v)=\tilde f(v,\ast,\ast)$, is a $H_7^ T$-invariant subspace of $V$. 
By Lemma \ref{invSub} either $\ker(\Psi)=W$, or $\ker(\Psi)=0$. 
If $\ker(\Psi)=W$, then $\tilde f$ is the
zero form. It follows that $\Psi$ is
injective.
Set $z=\tilde x_7\tilde y_7$ and take 
$\mathcal{B}$ as above.
We have $z=z^ T$ and
$$\tilde f( z^ {-1}v_i,z^ {-1} v_j,z^ {-1}v_k)=\eps^{i+j+k}\tilde f(v_i,v_j,v_k)=\tilde
f(v_i,v_j,v_k),$$
whence either $i+j+k=0$ or $\tilde f(v_i,v_j,v_k)=0$. 
It follows $f(v_i,v_j,v_k)=0$ except, possibly, for 
$(i,j,k)\in \{(0,\ell,-\ell), (1,2,-3),(-1,-2,3) \mid \ell=1,2,3\}$.
By the assumption that $\tilde f|_W$ is the zero-form, 
we have $\tilde f(v_{\pm 1},v_{\pm 2},v_{\mp 3})=0$.
If $\tilde f(v_0,v_i,v_{-i})=0$, then $\tilde f(v_i,\ast,\ast)$ is the zero form, in contrast 
with the injectivity of $\Psi$. So, $\tilde f(v_0,v_i,v_{-i})\neq 0$ for all $i$'s. 
Substituting each $v_i$ with a scalar multiple, if necessary, we get
$\tilde f=\sum_{i=1}^3\left( v_0^\ast\wedge
v_i^\ast\wedge v_{-i}^\ast\right)$. We conclude that  $d_{\Lambda^3(V^\ast)}^{H_7^T}=
1$.
\hfill $\square$ \end{proof}

\medskip

\begin{proof}[Proof of Theorem \ref{main}]
Suppose that $\char \F=2$ and let $H_6$ 
be an irreducible subgroup of $\SL_6(\F)$ generated by a non-rigid 
$(2,3,7)$-triple $(x_6,y_6,x_6y_6)$. Its similarity invariants are
those in \eqref{inv} by \cite{TV2}. Setting $K=H_6^ T$ in Proposition \ref{notril6}, 
we obtain that the transpose $H_6^ T$ of $H_6$ does not fix non-zero alternating trilinear forms. 
Hence, Proposition \ref{Scott} implies that $H_7=\langle \tilde x_7,\tilde y_7\rangle$, in the notation \eqref{shape7}, fixes a non-zero alternating trilinear form $f$  on $V=\F^ 7$.
We want to show that $f$ is similar to the Dickson form, applying \cite[Theorem 5(5)]{A}.
Its hypothesis require that $H_7$ acts indecomposably on $V=\F^ 7$, which is true
by Lemma \ref{orto}, and that $\rad(B_{v_0})=\langle v_0\rangle$, where
$B_{v_0}$ is the symmetric form defined as $B_{v_0}(u,w)=f(v_0,u,w)$. It remains to show this fact.
Now, the radical of $B_{v_0}$ is an $H_7$-invariant subspace of $V$. 
Thus, by Lemma \ref{invSub}, either $\rad(B_{v_0})=\langle v_0\rangle$ or $\rad(B_{v_0})=V$. 
Suppose the latter case holds, i.e.  $B_{v_0}$ is the zero form.

Take $\bar u,\bar v,\bar w\in \overline{V}=V/\langle v_0\rangle$ and define
$$\ovT(\bar u,\bar v,\bar w)=f(u, v, w),$$
where
$u=\eta v_0+\overline{u}$, $v=\lambda v_0+\overline{v}$ and $w=\mu v_0+\overline{w}$.
The function $\ovT$ is well-defined as
$$
f(u,v,w)  =f(\eta v_0+\overline{u},\lambda v_0+\overline{v},\mu v_0+\overline{w} )=
f(\overline{u},\overline{v},\overline{w}).$$
For every element of $h\in H_7$,
$\ovT(\overline{u},\overline{v},\overline{w})= f(u,v,w)=
f(h_6^ {-1}\overline{u},h_6^ {-1}\overline{v},h_6^{-1}\overline{w})$
 and so
$H_6$ fixes a non-zero alternating trilinear form $\ovT$ on $\overline{V}$, in contrast with Proposition \ref{notril6} 
(this time taking $K=H_6$).
We conclude that $H_7$
fixes a Dickson form by \cite[Theorem 5(5)]{A}, hence it is a subgroup of $G_2(\F)$.
\hfill $\square$
\end{proof}

\section{Hurwitz generators for $G_2(q)$}

In this section we set $q=p^a$, where $p$ is a prime, and consider $\F$ as the algebraic closure of the finite field $\F_q$. Our aim is to find explicit Hurwitz generators for the groups $G_2(q)$ when they exist, namely for $q\geq 5$ \cite{Ma}.

We first consider the case where $q$ is even. So let $p=2$ and $q\geq 8$. Our generators $x_6$, $y_6$ and their product $x_6y_6$ have similarity invariants \eqref{inv}, 
hence are conjugate to the matrices in \eqref{shapex6}, even if they have different shapes.  Indeed they are obtained from family (IIa) in \cite{V3} for special values of the parameters.
The choice of this family was made for  uniformity with the Hurwitz  
generators of $\PSp_6(q)$, $5\le q$ odd, 
used in  \cite{TV}. Indeed, it turns out that matrices of the same shape generate $\PSp_6(q)$ for $q$ odd and its subgroup $G_2(q)$ for $q$ even.  

For each  $r\in \F_q\setminus \F_4$ we 
set 
\begin{equation}\label{abc}
a=\frac{r + 1}{d},\quad
b=\frac{r^3 + r^2+ 1}{d},\quad c=\frac{r^3 + 1}{d},\quad d=r^2+r+1,
\end{equation}
\begin{equation}\label{x,y}
 x_6= \begin{pmatrix}
  0 & 0 & 1 & 0 & r & 1 \\
  0 & 0 & 0 & 1 & 0 & a \\
  1 & 0 & 0 & 0 & 1 & r \\
  0 & 1 & 0 & 0 & a & 0 \\
  0 & 0 & 0 & 0 & 0 & 1   \\
  0 & 0 & 0 & 0 & 1 & 0
    \end{pmatrix},
 \qquad
   y_6=\begin{pmatrix}
  1 & 0 & 0 & 0 & 1 & c \\
  0 & 1 & 0 & 0 & b & 1 \\
  0 & 0 & 0 & 0 & 1 & 0 \\
  0 & 0 & 0 & 0 & 0 & 1 \\
  0 & 0 & 1 & 0 & 1 & 0 \\
  0 & 0 & 0 & 1 & 0 & 1
   \end{pmatrix}
\end{equation}
and  define $H=\langle x_6,y_6\rangle$.

\begin{lemma}\label{lemma: irreducibility} 
The group $H$ is absolutely irreducible, except when:
$$r^{12} + r^9 + r^5 +  r^2 +1=0.$$
In particular, if $H$ is absolutely irreducible, then it is a subgroup of $G_2(q)$.
\end{lemma}

\begin{proof}
For the absolute irreducibility of $H$ apply \cite[Lemma 2.1]{TV} to the case  
$8 \le q$ even, $r\in \F_q\setminus \F_4$. For the second claim of the statement, apply Theorem \ref{main}. 
\hfill $\square$ \end{proof}

Now we want to exclude that $H$ is contained in a maximal subgroup of $G_2(q)$. 
By \cite{C}, the absolutely irreducible maximal subgroups $M$ of $G_2(q)$ are of type:
$$M\cong \SL_3(q).2,\quad M\cong \SU_3(q^ 2).2,\quad M\cong G_2(q_0)\;\;(\F_{q_0}<\F_q).$$

\begin{lemma}\label{PSU(3,q)} 
If $H$ is absolutely irreducible, it is not contained in $\SL_3(q)$ or $\SU_3(q^ 2)$.
\end{lemma}

\begin{proof}
If our claim is false,  $H$ arises from an action of $\SL_3(\F)$ 
on the symmetric square $S(\F^6)$ (see \cite[5.4.11]{KL}). But, in this action, an involution has similarity invariants $t+1$, $t+1$, $t^2+1$, $t^2+1$,
different from those of $x_6$.
\hfill $\square$ \end{proof}

\begin{rem}\label{F1} 
The \emph{field of definition} of a subgroup $H$ of $\GL_n(q)$  is the smallest subfield $\F_{q_1}$
of $\F_q$ such that a conjugate of $H$, under $\GL_n(\F)$, is contained in $\GL_n(q_1)(\F^\ast I)$, where $\F^\ast I$ denotes the center of $\GL_n(\F)$. This is to ensure that no conjugate of the projective image of $H$ lies in $\PGL_n(q_0)$ for some $q_0<q_1$. Clearly, when $H=H'$ is perfect, 
$\F_{q_1}$ coincides with the smallest subfield of $\F_q$ such that a conjugate of $H$ is contained in $\SL_n(q_1)$,
the derived subgroup of  $\GL_n(q_1)\left(\F^*I\right)$.
\end{rem}

\begin{theorem}\label{G2Even}
Let $x_6,y_6$ be as in \eqref{x,y} with $r\in \F_{q}\setminus \F_4$, $8\le q$ even, such that:
\begin{itemize}
\item[{\rm (i)}] $r^{12} + r^9 + r^5 +  r^2 +1\neq 0$;
\item[{\rm (ii)}] $\F_q=\F_2[ r^2+r]$.
\end{itemize}
Then $H=\langle x_6,y_6\rangle=G_2(q)$ and there exists
$r\in \F_{q}\setminus \F_4$ satisfying {\rm (i)} and {\rm (ii)}.
\end{theorem}

\begin{proof}
Condition (i) implies that $H$ is absolutely irreducible (Lemma \ref{lemma: irreducibility}) and by the same
lemma we have that $H\leq G_2(q)$.  Let $M$ be a maximal subgroup of $G_2(q)$ which contains $H$. 
Since $H$ is a Hurwitz group, it is perfect and so it is contained in the derived subgroup $M'$ of $M$. 
By  \cite[Lemma 3.2]{TV}, the minimal field of definition of $H$ is $\F_2[ r^2+r]$. 
Thus Condition (ii) gives that $M'\neq G_2(q_0)$ for any $q_0$ such that $\F_{q_0}$ is a
proper subfield of $\F_q$.
Moreover $M'\not\in \{\SL_3(q),\SU_3(q^2)\}$ by Lemma \ref{PSU(3,q)}. It follows that $H=G_2(q)$.
We now prove that, for $q>4$, there exists $r\in \F_{q}\setminus \F_4$, where $q=2^a$, satisfying Conditions (i) and (ii).
For each $\alpha\in \F_q$ such that $\F_2[\alpha]\neq \F_q$ there are at most two values of $r$ such that $r^ 2+r=\alpha$.
Let $N(a)$ be the number of elements $r\in \F_{2^a}$ such that $\F_2[r^ 2+r]\neq \F_{2^a}$.
Considering the possible subfields of $\F_{2^a}$, we get
$$N(a)\leq 2\left(2+2^ 2+2^3+\ldots+2^ {\lfloor\frac{a}{2} \rfloor}\right)= 2^ {2+\lfloor\frac{a}{2} \rfloor}-4.$$ 
Next, observe that if $a\geq 5$, then 
$\left(2^ {2+\lfloor\frac{a}{2} \rfloor}-4 \right)+12<2^ a$.
It follows that for $a\geq 5$, there exists $r\in \F_{2^a}\setminus \F_4$ satisfying Conditions (i) and (ii).
Finally for $q=8,16$, it suffices to take as $r$ a generator of $\F_{q}^ \ast$.
\hfill $\square$ \end{proof}
 
The following Remark gives Hurwitz generators of the Janko group $J_2$ over $\F_4$.

\begin{rem}
For $q=4$, consider the matrices (belonging to family (IIa) of \cite{V3})
\begin{equation}\label{J2}
x_{J_2}=\begin{pmatrix}
0 & 0 & 1 & 0 & \omega & 0 \\
0 & 0 & 0 & 1 & 1 & \omega^2\\
1 & 0 & 0 & 0 & 0 & \omega \\
0 & 1 & 0 & 0 & \omega^2 & 1 \\
0 & 0 & 0 & 0 & 0   & 1\\
0 & 0 & 0 & 0 & 1   & 0
    \end{pmatrix},\quad
y_{J_2}=\begin{pmatrix}
  1 & 0 & 0 & 0 & \omega^2 & \omega^2\\
  0 & 1 & 0 & 0 & \omega & \omega^2 \\
  0 & 0 & 0 & 0 & 1   & 0\\
  0 & 0 & 0 & 0 & 0   & 1 \\
  0 & 0 & 1 & 0 & 1   & 0 \\
  0 & 0 & 0 & 1 & 0   & 1
  \end{pmatrix},
\end{equation}
where $\omega^2+\omega+1=0$.
By a Magma calculation the group $H=\left\langle x_{J_2} , y_{J_2}\right\rangle$ as in \eqref{J2}
 is isomorphic to $J_2$.
 \end{rem}

Now, we consider the case $q\geq 5$ odd. In this case  our generators have similarity invariants \eqref{inv7} hence are conjugate to the matrices in \eqref{shape7}, even if they have different shapes. Indeed we
take them from family (II) in \cite{TV2}. 

For each  $r\in \F_q$ define $H=\langle x_7,y_7\rangle$, where

\begin{equation}\label{7odd}
x_7=\begin{pmatrix}
 0 &  0 & 0 & 1 & 0 & 0 & -4\\
 0 &  0 & 0 & 0 & 1 & 0 & r \\
 0 &  0 & 0 & 0 & 0 & 1 & -3 \\
 1 &  0 & 0 & 0 & 0 & 0 & -4 \\
 0 &  1 & 0 & 0 & 0 & 0 & r \\
 0 &  0 & 1 & 0 & 0 & 0 & -3 \\
 0 &  0 & 0 & 0 & 0 & 0 & -1 
\end{pmatrix},\;
y_7=
\begin{pmatrix}
 1 & 0 & 0 & 0 & 1 & 0 & r+2\\
 0 & 1 & 0 & 0 & 2 & 0 & 2r+8\\
 0 & 0 & 1 & 1 & 0 & 0 & -4 \\ 
 0 & 0 & 0 & 0 &-1 & 0 & 0 \\
 0 & 0 & 0 & 1 &-1 & 0 & 0 \\
 0 & 0 & 0 & 0 & 0 & 0 & -1 \\ 
 0 & 0 & 0 & 0 & 0 & 1 & -1
\end{pmatrix}.
\end{equation}

\begin{lemma}\label{Irrodd}
The group $H$ is absolutely irreducible, except when:
$$d=r^2+15 r+ 100=0.$$
In particular, if $H$ is absolutely irreducible, then it is a subgroup of $G_2(q)$.
\end{lemma}

\begin{proof}
For the absolute irreducibility of $H$ see the proof of  \cite[Theorem 1, page 2131]{TV2}.
For the second claim of the statement, apply Theorem   \ref{p7}. 
\hfill $\square$ \end{proof}
We notice that $[x_7,y_7]=x_7y_7^{-1}x_7y_7$ has trace $3$ and characteristic polynomial 
\begin{equation}\label{char7}
\chi_{[x_7,y_7]}(t)=t^ 7 -3t^ 6  -(d-5) t^ 5 -(d+7) t^ 4 + (d+7) t^ 3 + (d-5)t^ 2+3t-1.
\end{equation}
Again we want  to exclude that $H$ 
is contained in a maximal subgroup of $G_2(q)$. For $q$ odd, the classification is due 
to P. Kleidman in \cite{K}. We summarize in Table \ref{list} the relevant information of 
\cite[Tables 8.41, 8.42 pages 397--398]{H}.

\begin{table}[!h]
\begin{tabular}{ll}
Subgroup & Conditions \\\hline 
$G_2(q_0)$ &  $\F_{q_0}$ is a subfield of $\F_q$\\
${2^ 3}^.\SL_3(2)$ &   $q=p$\\
$\PSL_2(13)$ &  $q=p\equiv \pm 1,\pm 3,\pm 4\pmod{13}$ or \\
& $q=p^ 2,\ p\equiv \pm 2, \pm 5,\pm 6\pmod{13}$\\
$\PSL_2(8)$ &  $q=p\equiv \pm 1\pmod 9$ or $q=p^ 3,\ p= \pm 2,\pm 4\pmod{9}$\\
$G_2(2)= U_3(3):2$ &   $q=p\ge 5$\\
$J_1$ &  $q=11$\\
$\PGL_2(q)$ &  $p\ge 7$, $q\ge 11$\\ 
$\SL_3(q):2$  &  $p=3$\\
$\SU_3(q^2):2$ & $p=3$\\
$^2G_2(3^a)$ &  $p=3$, $a$ odd
\end{tabular}
\caption{Maximal irreducible subgroups of $G_2(q)$, $q=p^a$, $p\ne 2$.} \label{list}
\end{table}

Since $H$ is a Hurwitz group, it is perfect. So if $H$ is contained in a
maximal subgroup $M$ then $H\le M'$.

\begin{lemma}\label{vari}
Assume $5\le q$ odd, $H$ absolutely irreducible. Then:
\begin{itemize} 
\item[{\rm (i)}] $H$ is never contained in any maximal subgroup $M$ isomorphic to one of the following:\enskip
$\PSL_2(13),\ G_2(2),\ {2^3}^.\PSL_3(2),\ J_1$;
\item[{\rm (ii)}] $H$ is contained in a maximal subgroup $M\cong \PSL_2(8)$ iff $q=17$ and $r=1$.
\end{itemize}
\end{lemma}

\begin{proof}
We proceed with a case by case analysis. If $H\le M$, then $(x_7,y_7,x_7y_7)$ must coincide with some
 $(2,3,7)$-triple $\left(g_2,g_3, g_2g_3\right)$ in $M$.
Since the commutator $[x_7,y_7]$ has trace $3$, we are interested in those triples such that $3=\tr{[g_2,g_3]}$. The possible values 
of these traces are read from the ordinary and modular character tables of $M$.

\begin{itemize}
\item[(\rm{a})] Assume  $H\le M\cong\PSL_2(13)$, with the restrictions on $q$ given in  Table \ref{list}: in particular $p\ne 
3$.
In this case
the order of $[g_2,g_3]$ is either $6$ or $7$ or $13$ with  trace
$1$,  $0$ or $(1\pm \sqrt{13})/2$, respectively.
Equating these values to $3$ we get either $p=2$ or $p=3=q$, in contrast with our assumptions.

\item[(\rm{b})] Assume $H\le M\cong\PSL_2(8)$ with the restrictions on $q$ given in  Table \ref{list}. 
In this case $[g_2,g_3]$ has order $9$ and its trace is a root of $t^3-3t-1$. 
Equating its trace to $3$, we get $3^3-9-1=17=0$, whence $p=17=q$.
On the other hand, for $q=17$, the element $[x_7,y_7]$ has order $9$ only when $r=1$.
In this case $H\cong \PSL_2(8)$.

\item[(\rm{c})] Assume $H\le M\cong G_2(2)$,  $q=p\ge 5$. Since $H$ is perfect, we 
have that $H$ is actually contained in $M'=\SU_3(3)$. In this case $[g_2,g_3]$ has
order $4$. Taking $D=[x_7,y_7]^ 4$, we have $D_{7,7}=5r-7$.
If $p=5$, then $D_{7,7}=3\ne 1$. If $p\neq 5$, set
$r=\frac{8}{5}$. Then $D_{2,7}=-\frac{36}{5}\neq 0$, since $p\geq 7$.

\item[(\rm{d})] Assume $H\le M\cong {2^ 3}^.\PSL_3(2)$ with $q=p\geq 5$.
If $p\neq 7$, then the trace of $[g_2,g_3]$ is in $\{0,\pm 1\}$. As in 
(a) above we  get either $p=2$ or $p=3=q$, against our assumptions.
If $q=p=7$, then  $[x_7,y_7]$ has order $\geq 19$, so it is cannot be an element
of ${2^ 3}^ .\PSL_3(2)$.

\item[(\rm{e})] Assume $M\cong J_1$ with $q=11$.
The element $[g_2,g_3]$ has order $10$, $11$, $15$ or  $19$ with respective traces $9$, $7$, $5$ and $4$.
Since $\tr{[x_7,y_7]}=3$ we have a contradiction.
\end{itemize}
\hfill $\square$ \end{proof}

\begin{lemma}\label{SL3}
If $p=3$ and $H$ is absolutely irreducible, then it is not contained in $\SL_3(q)$ or $\SU_3(q^ 2)$.
\end{lemma}

\begin{proof}
The Hurwitz subgroups of $\SL_3(q^ 2)$ are isomorphic to $\PSL_2(7)$ or to $\PSL_2(27)$. 
However $\PSL_2(27)$ does not have irreducible representations of degree $7$.
If $H$ is  contained in $\PSL_2(7)$, then $[x_7,y_7]$ has order $4$, whence  
$\tr{[x_7,y_7]}=-1$, a contradiction.
\hfill $\square$ \end{proof}

\begin{lemma}\label{nctriples}
Assume $H$ absolutely irreducible. Then two triples $(x_7(r_1),y_7(r_1),$ $z_7(r_1))$ and $(x_7(r_2),y_7(r_2),z_7(r_2))$ 
are conjugate if, and only if, $r_1=r_2$. 
\end{lemma}

\begin{proof}
Assume that $(x_7(r_1),y_7(r_1),z_7(r_1))$ and $(x_7(r_2),y_7(r_2),z_7(r_2))$ are conjugate. 
Then the characteristic polynomial of $[x_7(r_1),y_7(r_1)]$ and $[x_7(r_2),y_7(r_2)]$ are the same, whence 
$r_1^ 2+15r_1=r_2^ 2+15r_2$. 
Suppose that $r_1\neq r_2$. Then $r_1=-15-r_2$.
Now, consider the element $w=[x_7,y_7]^ 2 y_7 (x_7y_7)^ 2 $ whose characteristic polynomial is
$$\chi_w(t)=t^ 7+f_1(r) t^ 6+f_2(r) t^ 5 - f_3(r)t^ 4 +f_3(r) t^ 3-f_2(r) t^ 2-f_1(r) t -1$$
where $f_1(r)=   r^3+25 r^2+250 r+999$.
From $f_1(r_1)=f_1(r_2)$ we get $d(2r_2+15)=0$. Since $H$ is absolutely irreducible, then
$d\neq 0$ and so
$r_2=-\frac{15}{2}$. However, in this case $r_1=-15-r_2=-\frac{15}{2}=r_2$, a contradiction.
We conclude that $r_1=r_2$.
\hfill $\square$ \end{proof}

As mentioned in the Introduction, the group $\PSL_2(q)$ is Hurwitz when either $q=p\equiv 0,\pm 1 \pmod 7$ or
$q=p^3$ and $p\equiv \pm 2, \pm 3\pmod 7$.
Moreover, in $\textrm{Aut}(\PSL_2(q))=\PGamL_2(q)$, there are $3$ conjugacy classes of Hurwitz triples 
for $q=p\equiv \pm 1 \pmod 7$, only one otherwise.

\begin{theorem}
Let $q=p^a$, $p$ odd, $q\ge 5$, and let $x_7,y_7$ be as in \eqref{7odd} with $r \in \F_q$
such that  $r\ne 1$ if $q=17$ and $r^2+15 r +100\neq 0$. Assume further that $\F_q=\F_p[r]$.
Then one of the following holds:
\begin{itemize}
\item[{\rm (i)}] $H=G_2(q)$;

\item[{\rm (ii)}] $H=\PSL(2,q)$ with $q=p$ if $p\equiv 0,\pm 1 \pmod 7$,   
$q=p^3$ if $p\equiv \pm 2, \pm 3\pmod 7$.
\end{itemize}
Moreover, in case {\rm (ii)} it is always possible to choose $r$ such that  
$H= G_2(q)$.
\end{theorem}

\begin{proof}
The group $H$ is absolutely irreducible by  Lemma \ref{Irrodd} and the assumption $r^2+15 r +100\neq 0$ and so it is a subgroup of $G_2(q)$. 
Suppose that there exists a maximal subgroup $M$ of $G_2(q)$ containing $H$.
Condition $\F_q=\F_p[r]$ implies that $\F_{q}$ is the field of definition of $H$ (see \cite[Remark 6]{TV2}). Thus
$M\neq G_2(q_0)$ for every $q_0< q$. By Lemma \ref{vari} we may also exclude $M'\in \{\PSL_2(13), \PSL_2(8), G_2(2), 
{2^ 3}^.\SL_3(2)\}$.
When $q=3^a$, we may also exclude $M'\in \{\SL_3(q), \SU_3(q^2)\}$ by Lemma \ref{SL3}. If $a$ is odd, suppose $M\cong {}^2G_2(3^a)$. 
By \cite[Proposition 3.14]{B} two semisimple elements of $^2G_2(3^a)$
which have the same trace are conjugate. Now $x_7y_7$ has trace $0$ and, for $p=3$,
also $[x_7,y_7]$ has trace $0$. So $[x_7,y_7]^ 7$ must be a $3$-element and so has trace $1$. 
However this gives $r(r^2-1)=0$ and so $r\in \F_3$, a contradiction with the assumption $\F_q=\F_3[r]$.
Thus $H$ is not contained in 
any subgroup of  Table \ref{list}, except possibly in $\PGL_2(q)$. We conclude  that either $H=G_2(q)$ or $H\leq M'$ with $M\cong \PGL_2(q)$.
Assume that the latter case holds. By the classification of the Hurwitz subgroups of $\PSL_2(q)$ \cite{Mac} we have
$H\cong \PSL_2(q)$ with $q=p$ if $p\equiv 0,\pm 1 \pmod 7$,  $q=p^3$ if $p\equiv \pm 2, \pm 3\pmod 7$.

In $G_2(q)$ there is just one conjugacy class of maximal subgroups isomorphic to $\PGL_2(q)$. Thus, by the information
given  before the statement,
there are at most $3$ non-conjugate
Hurwitz triples in $G_2(q)$ which can generate $\PSL_2(q)$. By Lemma \ref{nctriples},
different values of $r$ give rise to non conjugate triples. So we have to exclude at most  $3$
values of $r$ to avoid $H\le \PGL_2(q)$. Clearly we have to exclude at most other  $2$ values of $r$ for the absolute irreducibility.
We also note that, if $H\le \PGL_2(q)$, then $q\ne 17$ since $17\equiv 3 \pmod 7$. Our last claim now follows from the inequalities:
$p\ge 7>2+3$ if $q=p\equiv 0,\pm 1 \pmod 7$, and  $p^3-p>2+3$  if $q=p^3$.
\hfill $\square$ \end{proof}

For sake of completeness, we  provide Hurwitz generators also for the Janko  group $J_1$.
Consider the following matrices $x_{J_1},y_{J_1}\in \SL_7(11)$ (Family (I) of \cite{TV2}, with $r_3=2$ and $r_4=4$):
$$
x_{J_1}=\begin{pmatrix}
 0 &  0 &  0 &  1 &  0 &  0 & -1\\
 0 &  0 &  0 &  0 &  1 &  0 &  5\\
 0 &  0 &  0 &  0 &  0 &  1 &  2\\
 1 &  0 &  0 &  0 &  0 &  0 & -1\\
 0 &  1 &  0 &  0 &  0 &  0 &  5\\
 0 &  0 &  1 &  0 &  0 &  0 &  2\\
 0 &  0 &  0 &  0 &  0 &  0 & -1
\end{pmatrix}, \quad
y_{J_1}=\begin{pmatrix}
 1 &  0 &  0 &  0 &  4 &  0 &  3\\
 0 &  1 &  0 &  0 &  8 &  0 & -1\\
 0 &  0 &  1 &  1 &  0 &  0 &  9\\
 0 &  0 &  0 &  0 & -1 &  0 &  0\\
 0 &  0 &  0 &  1 & -1 &  0 &  0\\
 0 &  0 &  0 &  0 &  0 &  0 & -1\\
 0 &  0 &  0 &  0 &  0 &  1 & -1
\end{pmatrix}.
$$
Again, by a Magma calculation, the group $H=\left\langle x_{J_1} , y_{J_1}\right\rangle$  is isomorphic to $J_1$.

\end{document}